\documentclass{birkjour}
\usepackage[utf8]{inputenc}
\usepackage{amsmath}
\usepackage{amsfonts}
\usepackage{amssymb}
\usepackage{amsthm}
\usepackage{color}

 \newtheorem{thm}{Theorem}[section]
 \newtheorem{cor}[thm]{Corollary}
 \newtheorem{lem}[thm]{Lemma}
 \newtheorem{prop}[thm]{Proposition}
 \theoremstyle{definition}
 \newtheorem{defn}[thm]{Definition}
 \theoremstyle{remark}
 \newtheorem{rem}[thm]{Remark}
 \newtheorem{conj}{Conjecture}
 \newtheorem{ex}{Example}
 \numberwithin{equation}{section}
 
\newcommand{\vertiii}[1]{{\left\vert\kern-0.25ex\left\vert\kern-0.25ex\left\vert #1
  \right\vert\kern-0.25ex\right\vert\kern-0.25ex\right\vert}}

\newcommand{\Span}{\operatorname{span}}
\newcommand{\im}{\operatorname{Im}}
\newcommand{\re}{\operatorname{Re}}
\newtheorem{question}{Question}

\begin{document}

%-------------------------------------------------------------------------
% editorial commands: to be inserted by the editorial office
%
%\firstpage{1} \volume{228} \Copyrightyear{2004} \DOI{003-0001}
%\seriesextraline{This is the Concrete Title of this Book\br H.E. R and S.T.C. W, Eds.}
%\firstpage{1}

\title[Berger-Coburn theorem and the Toeplitz algebra]
 {Berger-Coburn theorem, localized operators, and the Toeplitz algebra}

%----------Author 1
\author[Wolfram Bauer]{Wolfram Bauer}

\address{%
Institut f\"{u}r Analysis\\
Welfengarten 1\\
30167 Hannover\\
Germany}

\email{bauer@math.uni-hannover.de}

%\thanks{}
%----------Author 2
\author{Robert Fulsche}
\address{Institut f\"{u}r Analysis\br
Welfengarten 1\br
30167 Hannover\br
Germany}
\email{fulsche@math.uni-hannover.de}


\date{January 1, 2004}
\dedicatory{Dedicated to Nikolai Vasilevski on the occasion of his $70^{th}$ birthday.}
%%% ----------------------------------------------------------------------

\begin{abstract}
We give a simplified proof of the Berger-Coburn theorem on the boundedness of Toeplitz operators and extend one part of this theorem to 
the setting of $p$-Fock spaces $(1\leq p \leq \infty)$. We present an overview of recent results by various authors on the compactness 
characterization via the Berezin transform for certain operators acting on the Fock space. Based on these results we present 
three new characterizations of the Toeplitz $C^*$ algebra generated by Toeplitz operators with bounded symbols. 
\end{abstract}

%%% ----------------------------------------------------------------------
\maketitle
%%% ----------------------------------------------------------------------
%-------------------------------------------
%%%%%%%%%%%%%%%%%%%%%%%%%%%%%%%%%%%%%%%%%%%%%%%%%%%%%%%%%%%%%%%%%%%%%%%%%%%%%%%%%%%%%%%%
\section{Introduction}
The present paper combines a survey part with some new results in the area of Toeplitz operators on Fock and Bergman spaces. They are among the most intensively 
studied concrete operators on function Banach or Hilbert spaces. Basic questions concern boundedness and compactness criteria, 
membership in operator ideals or their index and spectral theory. For a list of classical and more recent results we refer to  \cite{Axler_Zheng1998,Berger_Coburn1987,Berger_Coburn1994,Hagger2017,Hagger2019,Zhu2007,Zhu} and the literature cited therein. Instead of considering single operators,
the study of $C^*$ or Banach algebras generated by Toeplitz operators with specific types of symbols has attracted attention and essential progress has been made during the last years  \cite{Bauer_Vasilevski_2013,Quiroga_Barranco_Vasilevski_2007,Suarez2007,Vasilevski_2008,Xia}. 
\vspace{1mm}\par
On the one hand the purpose of this paper is to present an outline of some known classical and recent results on boundedness and compactness of Toeplitz operators on $p$-Fock spaces. 
Moreover, we highlight some relations between them that became apparent by applying more recent observations. On the other hand we add some new aspects to the theory. We present 
a simplified proof of one part of  the Berger-Coburn theorem from \cite{Berger_Coburn1994}. Our proof uses little ``machinery'' and even extends the boundedness criterion of the theorem 
to the setting of $p$-Fock spaces or weighted $L^p$-spaces  for any $1 \leq p \leq \infty$. In the second part we discuss compactness characterizations of bounded operators on Bergman 
and Fock spaces via the Berezin transform starting from the classical result by S. Axler and D. Zheng \cite{Axler_Zheng1998} in the case of the Bergman space over the unit disc and its 
Fock space version 
in \cite{Englis1999}. 
\begin{thm}[S. Axler, D. Zheng \cite{Axler_Zheng1998}]\label{Theorem_Axler_Zheng_Introduction}
Let $A$ be an element of the non-closed algebra generated by Toeplitz operators with bounded symbols acting on the Bergman space over the unit disc. Then $A$ is compact if and only if its 
Berezin symbol $\tilde{A}$ vanishes at the boundary. 
\end{thm}
Subsequently Theorem \ref{Theorem_Axler_Zheng_Introduction} has been extended to larger operator algebras, such as operators acting on a scale of Banach spaces in the Fock space 
when $p=2$ \cite{Bauer_Furutani}, $C^*$ algebras generated by sufficiently localized operators \cite{Xia_Zheng}, or the full Toeplitz algebra. The latter is the $C^*$ algebra generated 
by all Toeplitz operators with essentially bounded symbols \cite{Bauer_Isralowitz2012,Mitkovski_Suarez_Wick2013,Suarez2007}. 
Surprising identities between these algebras or their completions in \cite{Xia} show that these generalizations in fact are manifestations of the same result. More precisely, 
they just provide different characterizations of the Toeplitz algebra. 
%We would like to mention that Theorem \ref{Theorem_Axler_Zheng_Introduction} can also be proved in case of an 
%algebra $\mathcal{C}$ of compressions of band-dominated operators from weighted $L^2$-spaces to Bergman or Fock spaces \cite{Fulsche_Hagger,Hagger2017,Hagger2019}. Since this algebra contains the 
%Toeplitz algebra it is an interesting open problem whether $\mathcal{C}$ coincides with the Toeplitz algebra as well, or, whether it leads to a true extension of Theorem 
%\ref{Theorem_Axler_Zheng_Introduction} to a yet larger space.

In the Fock space setting and $p=2$ we prove three new characterizations of the Toeplitz algebra. One of the results (Corollary \ref{Corollary_Coburn_Isralowitz_Li}) involves bounded 
Toeplitz operators with (in general unbounded) symbols in the space $\textup{BMO}$ of functions having bounded mean oscillation. This observation links the discussion to the first 
part of the paper. In particular, Corollary \ref{Corollary_Coburn_Isralowitz_Li} gives an extension of a compactness characterization in \cite{Coburn_Isralowitz_Li2011} from single 
Toeplitz operators to elements in the generated algebra. 

Recently there have been great advances by combining ideas from operator theory on function spaces with techniques 
that originally have been developed for the spectral theory of band and band-dominated operators (see \cite{Bauer_Isralowitz2012,Fulsche_Hagger,Hagger2017,Hagger2019,Isralowitz_Mitkovski_Wick,Suarez2007}). This has  provided efficient tools in the analysis. We will discuss how the space of band-dominated operators gives rise to a new characterization of the Toeplitz algebra.
Throughout the text we have collected various open questions which we believe are interesting and may be subject of a future work.  

The paper is organized as follows. 
In Section \ref{Section_2} we introduce definitions and notation. As for a comprehensive source on the analysis of Fock spaces we refer to the textbook \cite{Zhu}.

A simple proof of an upper bound for the norm $\|T_f^t\|$ of a Toeplitz operator $T_f^t$ with (possibly unbounded) symbol $f$ acting on $p$-Fock space or weighted $L^p$-space 
is given in Section \ref{Section_3}. This result generalizes a theorem in \cite{Berger_Coburn1994}. We recall examples of Toeplitz operators with 
highly oscillating unbounded symbols  which are well-known in the literature (e.g. \cite{Berger_Coburn1994}) and illuminate the result. Finally, we comment on some progress in 
\cite{Bauer_Coburn_Isralowitz2010} on a conjecture by C. Berger and L. Coburn in \cite{Berger_Coburn1994} concerning a boundedness characterization of Toeplitz operators. 

Section \ref{Section_4} starts with a (certainly non-complete) survey of the literature on compactness characterizations via the Berezin transform. In particular, we relate the results 
by applying surprising characterizations of the Toeplitz algebra $\mathcal{T}$ in \cite{Xia}. By combining some of these results and using an inequality of Section \ref{Section_2} we can provide 
three new characterizations of $\mathcal{T}$. In particular, our observation generalizes a theorem in \cite{Coburn_Isralowitz_Li2011} on a compactness characterization of 
Toeplitz operators with $\textup{BMO}$-symbols. 
%%%%%%%%%%%%%%%%%%%%%%%%%%%%%%%%%%%%%%%%%%%%%%%%%%%%%%%%%%%%%%%%%%%%%%%%%%%%%%%%%%%%%%%%%%
\section{Preliminaries}
\label{Section_2}
%%%%%%%%%%%%%%%%%%%%%%%%%%%%%%%%%%%%%%%%%%%%%%%%%%%%%%%%%%%%%%%%%%%%%%%%%%%%%%%%%%%%%%%%%%
On $\mathbb C^n$ consider the one-parameter family of probability measures
\[ d\mu_t(z) = \frac{1}{(\pi t)^n} e^{-\frac{1}{t}|z|^2}dV(z), \hspace{4ex} t>0\]
where $V$ denotes the usual Lebesgue measure on $\mathbb C^n\cong \mathbb{R}^{2n}$. Here we write $|z|= (|z_1|^2+ \ldots +|z_n|^2)^{\frac{1}{2}}$ for the 
Euclidean norm of $z \in \mathbb C^n$. Throughout the paper we write $\mathbb{N}_0= \{0,1,2,\ldots\}$ for the non-negative integers. Let $1 \leq p < \infty$ and $t > 0$ and define
\[ L_t^p := L^p(\mathbb C^n, \mu_{2t/p})  \hspace{4ex} \mbox{\it and} \hspace{4ex} F_t^p := L_t^p \cap \textup{Hol}(\mathbb C^n).  \]

Here $\textup{Hol}(\mathbb{C}^n)$ denotes the space of holomorphic functions on $\mathbb{C}^n$. Further, for measurable $g: \mathbb C^n \to \mathbb C$ we use the notation
\[ \| g\|_{L_t^\infty} = \underset{z \in \mathbb C^n}{\operatorname{ess ~sup}} | g(z) e^{-\frac{1}{2t}|z|^2}| \]
and set
\begin{align*}
L_t^\infty &:= \big{\{} g: \mathbb C^n \to \mathbb C\: : \:  g \text{ \it measurable,  } \| g\|_{L_t^\infty} < \infty \big{\}}, \\
F_t^\infty &:=L_t^\infty \cap \textup{Hol}(\mathbb{C}^n). 
\end{align*}
Recall that $F_t^2$ is a reproducing kernel Hilbert space equipped with the standard $L_t^2$-inner product 
$\langle f,g \rangle_t:= \int_{\mathbb{C}^n} f(z) \overline{g(z)} d\mu_t(z)$ for $f,g \in F_t^2$ and reproducing kernel
\[ K_z^t(w) = K^t(w,z) = e^{\frac{w \cdot \overline{z}}{t}} \in F_t^2. \]
In particular, the orthogonal projection $P^t: L_t^2 \to F_t^2$ is given by
\[ P^tf(z) = \big{\langle} f, K_z^t\big{\rangle}_t = \int_{\mathbb{C}^n} f(w) e^{\frac{z \cdot \overline{w}}{t}} d\mu_t(w). \]
For $p \neq 2$, the integral operator
\begin{align*}
P^tf(z) &= \int_{\mathbb{C}^n} f(w)e^{\frac{z \cdot \overline{w}}{t}}d\mu_t(w)\\
&= \Big ( \frac{p}{2} \Big )^n \int_{\mathbb{C}^n} f(w) e^{\frac{z\cdot \overline w}{t}} e^{\frac{1}{t}(\frac{p}{2} - 1)|w|^2}d\mu_{2t/p}(w)
\end{align*}
still defines a bounded projection $P^t: L_t^p \to F_t^p$ \cite[Theorem 7.1]{Janson_Peetre_Rochberg}. Given a suitable measurable function $f: \mathbb{C}^n \to \mathbb C$ 
and any $t>0$ we define the Toeplitz operator $T_f^t: F_t^p \to F_t^p$ by
\[ T_f^t = P^t M_f. \]
If not further specified we will consider $T_f^t$ on the domain 
$$D(T_f^t) = \big{\{} h \in F_t^p \: : \:  ~ fh \in L_t^p\big{\}}.$$ 
\par 
In the case $t = 1$ we shortly write $P = P^1$ and $T_f = T_f^1$. Especially for $p=2$ Toeplitz operators on the Fock space are well-studied under different aspects 
(see \cite{Isralowitz_Zhu2010, Zhu} and the literature therein). However, due to a number of open problems, in particular concerning their algebraic and analytic properties, 
they remain interesting objects of current research. 

Denote by
\begin{equation}\label{Defn_normalized_reproducing_kernel}
k_z^t(w) = \frac{K^t(w,z)}{\sqrt{K^t(z,z)}}, \hspace{4ex}\hspace{3ex} \mbox{\it where} \hspace{3ex}  z,w \in \mathbb{C}^n
\end{equation}
the normalized reproducing kernel. Again, for $t=1$ we write $k_z = k_z^1$ and $K_z = K_z^1$. We define the {\it Berezin transform} of an operator $A \in \mathcal L(F_t^2)$ by
\begin{equation}\label{Definition_Berezin_transform_Preliminaries}
 \widetilde{A}^{(t)}(z) := \big{\langle} Ak_z^t, k_z^t \big{\rangle}_t  
  \end{equation}
and we use the short notation $\widetilde{A}= \widetilde{A}^{(1)}$. Note that (\ref{Definition_Berezin_transform_Preliminaries}) defines a complex valued function on $\mathbb{C}^n$ 
which is bounded whenever $A$ is a bounded operator. The Berezin transform of a measurable function $f$ with the property that $fK^t_z \in L_t^2$ for all $z \in \mathbb C^n$ is 
denoted by
\[ \widetilde{f}^{(t)}(z) := \big{\langle} f k_z^t, k_z^t\big{\rangle}_t. \]
$\widetilde{f}^{(4t)}$ coincides with the {\it heat transform} of $f$ on $\mathbb{C}^n$ at time $t$. Observe that $\widetilde{T_f^t}^{(t)} = \widetilde{f}^{(t)}$.
%%%%%%%%%%%%%%%%%%%%%%%%%%%%%%%%%%%%%%%%%%%%%%%%%%%%%%%%%%%%%%%%%%%%%%%%%%%%%%%%%%%%%%%%%
\section{On the Berger-Coburn theorem}
\label{Section_3}
%%%%%%%%%%%%%%%%%%%%%%%%%%%%%%%%%%%%%%%%%%%%%%%%%%%%%%%%%%%%%%%%%%%%%%%%%%%%%%%%%%%%%%%%%
One of the simple properties of Toeplitz operators is the fact that boundedness of the symbol implies boundedness of the operator. The converse in general is false. 
Indeed, one of the important questions in the theory of Toeplitz operators is a characterization of the boundedness of the operator in terms of its (unbounded) symbol. 
In the case $p=2$ we recall the classical Berger-Coburn theorem on the boundedness of Toeplitz operators on the Fock space:
\begin{thm}[Berger-Coburn \cite{Berger_Coburn1994}]\label{thm1}
Assume that $f \in L_t^2$. Then the following norm estimates hold true for $T_f^t: F_t^2 \to F_t^2$:
\begin{align*}
C(s) \| T_f^t\| &\geq \| \widetilde{f}^{(s)}\|_\infty, \quad 2t > s > t/2\\
c(s) \| \widetilde{f}^{(s)}\|_\infty &\geq \| T_f^t\|, \quad t/2 > s > 0. 
\end{align*}
Here, $C(s), c(s) > 0$ are universal constants depending only on $s, t$ and $n$.
\end{thm}
Berger and Coburn proved this result for $t = 1/2$. However, the proof directly generalizes to the case $t>0$. We start by a short outline of the original proof of Theorem 
\ref{thm1} in \cite{Berger_Coburn1994}: 
\vspace{1ex}\par 
The first estimate is obtained by a trace-estimate of an operator product. More precisely, in  \cite{Berger_Coburn1994} a trace-class operator $S_a^{(s)}$ is constructed 
depending on $a \in \mathbb{C}^n$ and $s$ in the above range such that 
%for which it holds 
$$\operatorname{trace}\big{(}T_f^t S_a^{(s)}\big{)} = \widetilde{f}^{(s)}(a).$$ 
Then the standard trace estimate 
$$|\operatorname{trace}(AB)| \leq \|A\| \|B\|_{\textup{tr}}$$ 
can be applied, where $A$ is a bounded operator, $B$ a trace class operator and $\| \cdot \|_{\textup{tr}}$ denotes the {\it trace norm}. The second inequality provides a boundedness criterion for Toeplitz operators and its proof 
consists of the following steps:
%On the other hand, the proof for the second estimate consists of the following steps:
\begin{enumerate}
\item Transform $T_f^t$ into a Weyl-pseudodifferential operator on $L^2(\mathbb R^n)$ via the Bargmann transform $\mathcal B_t: F_t^2 \to L^2(\mathbb R^n)$, which defines a bijective 
Hilbert space isometry.
\item Estimate  the symbol of the pseudodifferential operator $W_{\sigma(f)}= \mathcal{B}_t T_f^t \mathcal{B}_t^{-1}$.
\item Apply the Calder\'{o}n-Vaillancourt Theorem.
\end{enumerate}
\par
Here we will present a proof of the second inequality based on simple estimates for integral operators. In particular, we avoid the theory of pseudodifferential operators and 
the application of the Calder\'{o}n-Vaillancourt Theorem. Instead of working on $L^2(\mathbb{R}^n)$ all calculations will be done in the Fock space setting. 
Our proof has also the advantage that it generalized to the case of the $p$-Fock space $F_t^p$ for $1\leq p  <\infty$. Moreover, it applies to Toeplitz operators interpreted as 
integral operators on the enveloping space $L_t^p$.
\vspace{1mm}\par 
In the following we assume that $f: \mathbb{C}^n \to \mathbb{C}$ is a measurable function such that $fK_z^t \in L^2(\mathbb{C}^n,\mu_t)$ for all $z \in \mathbb{C}^n$. In particular, the 
Berezin transform and its off-diagonal extension 
\[ \widetilde{f}^{(t)}(z,w) := \big{\langle} fk_z^t, k_w^t \big{\rangle}_t \]
exist for all $z,w \in \mathbb{C}^n$.

Without any further assumptions $T_f^t$ is an unbounded operator in general. Let $p = 2$ and note that $T_f^t$ is densely defined since $\{ K^t(\cdot, z) \: : \: z \in \mathbb C^n\}$ is 
a total set in $F_t^2$. Hence we can consider its adjoint $(T_f^t)^\ast$. Recall that% for $h \in F_t^2$ it holds
\[ h \in D\big{(}(T_f^t)^\ast\big{)} \Leftrightarrow \exists C > 0: |\langle T_f^tg, h\rangle_t|\leq C\|g\|_t ~ \forall g\in D(T_f^t). \]
For $z \in \mathbb{C}^n$ and $g \in D(T_f^t)$ it holds
\begin{align*}
|\langle T_f^tg, K_z^t\rangle_t| = |\langle g, \overline{f}K_z^t\rangle_t|\leq \| g\|_t \| f K_z^t\|_t,
\end{align*}
and therefore $\mbox{span} \{ K_z^t \: : \:  z \in \mathbb{C}^n\} \subseteq D((T_f^t)^\ast)$. In particular, the adjoint operator $(T_f^t)^\ast$ is densely defined as well. 

We compute a useful representation for the integral kernel of $T_f^t$ on $F_t^2$:
\begin{align}
T_f^tg(z) &= \langle T_f^t g, K_z^t\rangle_t \notag \\
&= \langle g, (T_f^t)^\ast K_z^t\rangle_t = \int_{\mathbb C^n} \overline{ \big ( (T_f^t)^\ast K_z^t\big ) (w)} g(w) d\mu_t(w).\label{GL_representation_Toeplitz_operator_as_integral}
\end{align}
This integral kernel can further be computed as\label{eq_integral_kernel}
\begin{align*}
\overline{\big ( (T_f^t)^\ast K_z^t \big ) (w) } &= \overline{\langle (T_f^t)^\ast K_z^t, K_w^t\rangle_t} = \langle K_w^t, (T_f^t)^\ast K_z^t\rangle_t\\
&= \langle T_f^t K_w^t, K_z^t\rangle_t = \langle f K_w^t, K_z^t\rangle_t\\
&= \sqrt{K^t(w,w) K^t(z,z)} \langle fk_w^t, k_z^t\rangle_t\\
&= e^{\frac{1}{2t}(|w|^2 + |z|^2)}\widetilde{f}^{(t)}(w,z)\\
&= e^{\frac{1}{2t}|z-w|^2 + \frac{1}{t}\re(z \cdot \overline{w})}\widetilde{f}^{(t)}(w,z).
\end{align*}
Inserting this expression above gives: 
\begin{equation}\label{Toeplitz_operator_as_an_integral_operator}
T_f^tg(z) =\int_{\mathbb{C}^n} e^{\frac{1}{2t}|z-w|^2 + \frac{1}{t}\re(z \cdot \overline{w})}\widetilde{f}^{(t)}(w,z)g(w) d\mu_t(w),  
\end{equation}
which can be considered as an integral operator either on $F_t^2$ or $L_t^2$. 
\vspace{1ex}\par 
Recall that $P^t$ is defined on $L_t^p$ for each $1 \leq p \leq \infty$ by the same integral expression, hence $T_f^t: F_t^p \to F_t^p$ and $T_f^t: F_t^2 \to F_t^2$ act in the same way on 
the space $\Span \{ K_z^t\: : \:  z \in \mathbb C^n\}$. Therefore, the integral kernel gives the same operator for the $F_t^p$-version of the Toeplitz operator on $F_t^p \cap F_t^2$, which is a dense subset of $F_t^p$ (for $p < \infty$).

We are now going to derive estimates for $\widetilde{f}^{(t)}(w,z)$. Easy computations show that  $fK_z^s \in L^2(\mathbb{C}^n, \mu_s)$ for $0 < s < t$.  
Hence $\widetilde{f}^{(s)}(w,z)$ exists for all $s$ in the range $(0, t]$. Moreover, from the {\it semigroup property} of the heat transform it follows that:
\begin{align*}
\big{\langle} fk_z^t, k_z^t\big{\rangle}_t = \widetilde{f}^{(t)}(z) = \Big ( \widetilde{f}^{(s)} \Big )^{\sim (t-s)}(z) = \big{\langle} \widetilde{f}^{(s)} k_z^{t-s}, k_z^{t-s}\big{\rangle}_{t-s}
\end{align*}
for all $0\leq s <t$. In particular,
\begin{equation}\label{eqsemigroup}
\big{\langle} fK_z^t, K_z^t \big{\rangle}_t = e^{-\frac{s}{t(t-s)}|z|^2} \big{\langle} \widetilde{f}^{(s)}K_z^{t-s}, K_z^{t-s}\big{\rangle}_{t-s}.
\end{equation}
We can extend this relation to off-diagonal values: 
\begin{lem}\label{lmm1}
For $z, w \in \mathbb{C}^n$ and $0 \leq  s < t$ it holds:
\begin{equation}\label{GL_heat_transform_off_diagonal}
\widetilde{f}^{(t)}(w,z) = e^{\frac{s}{2t(t-s)}|w-z|^2 - \frac{is}{t(t-s)}\im(z\cdot \overline{w})}\big{\langle} \widetilde{f}^{(s)} k_w^{t-s}, k_z^{t-s}\big{\rangle}_{t-s}. 
\end{equation}
\end{lem}
\begin{proof}
Recall that $\langle fK_w^t, K_z^t\rangle_t$ is anti-holomorphic in $w$ and holomorphic in $z$. The same holds for
\[ e^{-\frac{s}{t(t-s)}z\cdot \overline{w}} \big{\langle} \widetilde{f}^{(s)} K_w^{t-s}, K_z^{t-s}\big{\rangle}_{t-s}. \]
Since both functions agree on the diagonal $w = z$ by Equation (\ref{eqsemigroup}), they agree for all choices of $w,z \in \mathbb{C}^n$ by a well-known identity theorem \cite[Proposition 1.69]{Folland1989}. Hence,
\begin{align*}
\big{\langle} fK_w^t, K_z^t\big{\rangle}_t = e^{-\frac{s}{t(t-s)}z\cdot \overline{w}} \big{\langle} \widetilde{f}^{(s)} K_w^{t-s}, K_z^{t-s}\big{\rangle}_{t-s}.
\end{align*}
Division by the normalizing factors implies (\ref{GL_heat_transform_off_diagonal}). 
%\begin{align*}
%\widetilde{f}^{(t)}(w,z) &= e^{\frac{s}{2t(t-s)}|w-z|^2 - \frac{is}{t(t-s)}\im(z \cdot \overline{w})}\langle \widetilde{f}^{(s)} k_w^{t-s}, k_z^{t-s}\rangle_{t-s}.
%\end{align*}
\end{proof}
\begin{lem}\label{lmm2}
Let $g \in L^\infty(\mathbb C^n)$ and $t > 0$. Then, it holds
\[ \big{|} \big{\langle} gk_w^t, k_z^t\big{\rangle}_t\big{|} \leq \| g\|_\infty e^{-\frac{1}{4t}|w-z|^2}. \]
\end{lem}
\begin{proof} Let $z,w \in \mathbb{C}^n$, then:
\begin{align*}
\big{|}\big{\langle} gk_w^t, k_z^t\big{\rangle}_t\big{|} &= \frac{1}{\sqrt{K^t(w,w) K^t(z,z)}} \Big | \int_{\mathbb{C}^n} g(u) e^{\frac{1}{t}(u \cdot \overline{w} + z \cdot \overline{u})} d\mu_t(u)\Big | \\
&\leq e^{-\frac{1}{2t}(|z|^2 + |w|^2)} \| g\|_\infty \int_{\mathbb C^n} e^{\frac{1}{t}\re(u\cdot \overline{w} + z \cdot \overline{u})} d\mu_t(u)\\
&= e^{-\frac{1}{2t}(|z|^2 + |w|^2)} \| g\|_\infty \int_{\mathbb C^n} e^{\frac{1}{2t}u\cdot \overline{(w+z)} + \frac{1}{2t}\overline{u}\cdot(w+z)}d\mu_t(u)\\
&= e^{-\frac{1}{2t}(|z|^2 + |w|^2)} \| g\|_\infty \big{\langle} K_{(w+z)/2}^t, K_{(w+z)/2}^t\big{\rangle}_t\\
&= e^{-\frac{1}{2t}(|z|^2 + |w|^2)} \| g\|_\infty K^t((w+z)/2, (w+z)/2)\\
&= e^{-\frac{1}{2t}(|z|^2 + |w|^2) + \frac{1}{4t}|w+z|^2} \| g\|_\infty \\
&= \| g\|_\infty e^{-\frac{1}{4t}|w-z|^2}. 
\end{align*}
\end{proof}
\begin{thm}
Assume that $f$ is such that $\widetilde{f}^{(s)}$ is bounded for some $s \in (0,t/2)$. 
\begin{itemize}
\item[(i)] For $1 \leq p \leq \infty$, the integral operator 
\[ I_f^t: L_t^p \to F_t^p \]
defined by
\[ I_f^tg(z) := \int_{\mathbb{C}^n} e^{\frac{1}{2t}|z-w|^2 + \frac{1}{t}\re(z \cdot \overline{w})}\widetilde{f}^{(t)}(w,z) g(w) d\mu_t(w) \]
is bounded. Moreover, 
\[ \| I_f^t\|_{L_t^p \to F_t^p} \leq C \| \widetilde{f}^{(s)}\|_\infty \]
for some constant $C$ which only depends on $t, s$ and $n$. 
\item[(ii)]
In particular, for $1 \leq p < \infty$ the Toeplitz operator $T_f^t: F_t^p \to F_t^p$ is bounded with the same norm bound
\[ \| T_f^t\|_{F_t^p \to F_t^p} \leq C\| \widetilde{f}^{(s)}\|_\infty. \]
\end{itemize}
\end{thm}
\begin{proof}
We prove $(i)$ first. By Lemma \ref{lmm1} and Lemma \ref{lmm2} we obtain
\begin{align*}
|\widetilde{f}^{(t)}(w,z)| \leq \| \widetilde{f}^{(s)}\|_\infty e^{ \left( \frac{s}{2t(t-s)} - \frac{1}{4(t-s)}\right) |w-z|^2}.
\end{align*}
Set $\frac{s}{2t(t-s)} - \frac{1}{4(t-s)} = -\gamma_{s,t}$ and observe that
\[ \gamma_{s,t} > 0 \Leftrightarrow s < \frac{t}{2}. \]
%We are going to prove that the integral operator $I_f^t: F_t^p \to F_t^p$ acting as
%\[ I_f^tg(z) := \int_{\mathbb{C}^n} e^{\frac{1}{2t}|z-w|^2 + \frac{1}{t}\re(z \cdot \overline{w})}\widetilde{f}^{(t)}(w,z) g(w) d\mu_t(w) \]
%is bounded for each $1 \leq p \leq \infty$. As the operator agrees with $T_f^t$ on $F_t^p \cap F_t^2$ (which is dense in $F_t^p$ for $1 \leq p < \infty$), this will imply the result.

For $p = \infty$, it holds
\begin{align*}
\| I_f^tg\|_{F_t^\infty} &\leq \Big ( \frac{1}{\pi t} \Big )^n \| \widetilde{f}^{(s)}\|_\infty\times \\
& \times  \int_{\mathbb{C}^n} e^{\left(\frac{1}{2t}-\gamma_{s,t} \right)|z-w|^2 + \frac{1}{t}\re(z \cdot \overline w) - \frac{1}{2t}|z|^2 - \frac{1}{t}|w|^2}|g(w)| dV(w)\\
&= \Big (\frac{1}{\pi t} \Big )^n \| \widetilde{f}^{(s)}\|_\infty \int_{\mathbb{C}^n} |g(w)|e^{-\frac{1}{2t}|w|^2} e^{-\gamma_{s,t}|w-z|^2}dV(w)\\
&\leq \Big (\frac{1}{\pi t} \Big )^n \| \widetilde{f}^{(s)}\|_\infty \| g\|_{L_t^\infty} \int_{\mathbb{C}^n} e^{-\gamma_{s,t}|w-z|^2}dV(w)\\
&= \| \widetilde{f}^{(s)}\|_\infty \| g\|_{L_t^\infty} \Big ( \frac{1}{\gamma_{s,t} t} \Big )^n.
\end{align*}
This  proves
\[ \| I_f^t\|_{L_t^\infty \to F_t^\infty} \leq \| \widetilde{f}^{(s)}\|_\infty \Big ( \frac{1}{\gamma_{s,t} t} \Big )^n. \]
For $p = 1$, we obtain the following:
\begin{align*}
&\| I_f^t g\|_{F_t^1}
 = \int_{\mathbb C^n} |I_f^tg(z)|d\mu_{2t}(z)\\
&= \Big ( \frac{1}{2t^2 \pi^2} \Big )^n\times \\
&\times  \int_{\mathbb C^n} \Big | \int_{\mathbb C^n} e^{\frac{1}{2t}|z-w|^2 + \frac{1}{t} \re(z \cdot \overline w) - \frac{1}{t}|w|^2 - \frac{1}{2t}|z|^2}\widetilde{f}^{(t)} (w,z) g(w) dV(w) \Big | dV(z)\\
&\leq \Big ( \frac{1}{2 t^2 \pi^2} \Big )^n \int_{\mathbb C^n} \int_{\mathbb C^n} |g(w)| |\widetilde{f}^{(t)}(w,z)| e^{-\frac{1}{2t}|w|^2} dV(w) dV(z)\\
&\leq \Big ( \frac{1}{2 t^2 \pi^2} \Big )^n \| \widetilde{f}^{(s)}\|_\infty \int_{\mathbb C^n} \int_{\mathbb C^n} |g(w)| e^{-\frac{1}{2t}|w|^2} e^{-\gamma_{s,t}|w-z|^2} dV(w) dV(z)\\
&= \Big ( \frac{1}{2 t^2 \pi^2} \Big )^n \| \widetilde{f}^{(s)}\|_\infty \int_{\mathbb C^n} \int_{\mathbb C^n} e^{-\gamma_{s,t}|w-z|^2} dV(z) |g(w)| e^{-\frac{1}{2t}|w|^2} dV(w)\\
&= \Big ( \frac{1}{ \gamma_{s,t} t} \Big )^n \| \widetilde{f}^{(s)}\|_\infty  \| g\|_{L_t^1}.
\end{align*}
Therefore,
\[ \| I_f^t\|_{L_t^1 \to F_t^1} \leq \| \widetilde{f}^{(s)}\|_\infty \Big (\frac{1}{\gamma_{s,t} t}\Big )^n. \]
Using complex interpolation (i.e. the fact that  $[L_t^1, L_t^\infty]_\theta = L_t^{p_\theta}$ and $[F_t^1, F_t^\infty]_\theta = F_t^{p_\theta}$, where 
$p_\theta = 1/(1-\theta)$, c.f. \cite{Janson_Peetre_Rochberg, Zhu}), one obtains
\[ \| I_f^t\|_{L_t^p \to F_t^p} \leq \| \widetilde{f}^{(s)}\|_\infty \Big (\frac{1}{\gamma_{s,t} t}\Big )^n. \]

For (ii) observe that for $1 \leq p < \infty$, $T_f^t$ acts by the same integral expression as $I_f^t$, hence it holds $T_f^t = I_f^t|_{F_t^p}$. 
Therefore, the Toeplitz operator inherits the norm estimate from $I_f^t$.
\end{proof}
\begin{rem}
\begin{enumerate}
\item From the proof we see that the constant in the estimate essentially behaves as $\left( \frac{t}{t/2-s} \right)^n$ when $s \to t/2$. 
\item For $p = 2$ one can obtain a direct proof of the statement without interpolation, using Lemma \ref{lmm1} and \ref{lmm2} and Schur's test 
(cf. the proof of Lemma \ref{sufficiently_localized_is_bounded} below, which uses a similar argument).
\end{enumerate}
\end{rem}
There are various problems left open concerning the characterization of bounded operators. Here, we mention some of them:
\begin{question}
Does an $L_t^p$-version of the first estimate in Theorem \ref{thm1} hold?
\end{question}
\begin{question}
Can one give a reasonable characterization of boundedness for products of Toeplitz operators with unbounded symbols. 
\end{question}
Recall that for $p=2$ and under certain growth conditions 
of the symbols at infinity finite products of unbounded Toeplitz operators have been well-defined in \cite{Bauer2009}. They can be interpreted as elements in an algebra of operators acting on a scale 
of Banach spaces in the Fock space. We mention that boundedness of products $T_fT_{\bar{g}}$ on the $2$-Fock space with holomorphic symbols $f,g$ has recently been 
characterized in \cite{Cho_Park_Zhu}. 
\vspace{1ex}\par
The following conjecture in the case $t= \frac{1}{2}$ was made by C. Berger and L. Coburn in \cite{Berger_Coburn1994}.  
\begin{conj} 
$T_f^t$ is bounded if and only if $\widetilde{f}^{(\frac{t}{2})}$ is bounded. 
\end{conj}
There are indications that this conjecture may actually hold true. The authors of \cite{Berger_Coburn1994}  accompanied their conjecture with the following example:
\begin{ex}[\cite{Berger_Coburn1994}]\label{ex1}
Let $\lambda \in \mathbb{C}$ be a parameter. Consider the functions
\[ g_\lambda(z) := e^{\lambda |z|^2}, \quad z \in \mathbb C^n \]
for $\re \lambda < \frac{1}{2t}$. The latter condition guarantees that the heat transforms  $\widetilde{g_\lambda}^{(s)}$ exist for all $0 < s < 2t$. 
Moreover, since $g_\lambda$ is radial, a simple calculation shows that the Toeplitz operator $T_{g_{\lambda}}^t$ acts diagonally on the standard orthonormal basis 
$\{ e_m^t\: :\: ~m \in \mathbb N_0^n\}$ of $F_t^2$, where
\begin{equation}\label{Standard_ONB_Fock_space}
e_m^t(z) := \frac{z^m}{\sqrt{t^{|m|} m!}}. 
\end{equation}
Here, we used standard multiindex notation. One has
\[ T_{g_\lambda}^te_m^t = \left( 1 - t\lambda \right)^{-(|m|+n)} e_m^t. \]
This implies that $T_{g_\lambda}^t$ is bounded on $F_t^2$ if and only if $|1-t\lambda|\geq 1$.

Using the well-known formula
\[ \int_{\mathbb R} e^{bx - \frac{a}{2}x^2} dx = \sqrt{\frac{2\pi}{a}} e^{\frac{b^2}{2a}}, \quad a, b \in \mathbb C, ~\re(a) > 0, \]
one can show that
\[ \widetilde{g_\lambda}^{(s)}(w) = \frac{1}{(1-s\lambda)^n} e^{\frac{\lambda}{1-s\lambda}|w|^2}, \quad 0 < s < 2t. \]
This function is obviously bounded if and only if
\[ \re\left( \frac{\lambda}{1-s\lambda} \right) \leq 0 \Longleftrightarrow \left |1 - 2s\lambda \right | \geq 1. \]
In particular, $T_{g_\lambda}^t$ is bounded if and only if $\widetilde{g_\lambda}^{(\frac{t}{2})}$ is a bounded function. 
\end{ex}
There are two other known cases for which the conjecture has been verified. The first concerns operators with non-negative symbols.
\begin{prop}[{Berger-Coburn \cite{Berger_Coburn1987}}]
Let $f \geq 0$ be such that $f \in L_t^2$. Then, $T_f^t: F_t^2 \to F_t^2$ is bounded if and only if $\widetilde{f}^{(\frac{t}{2})}$ is bounded.
\end{prop}
Since this result is not directly stated in \cite{Berger_Coburn1987} (even though \cite{Berger_Coburn1994} refers this fact to that article), we give a short proof based on a lemma from that paper:
\begin{proof}
In the case of $t = 1/2$, \cite[Lemma 14]{Berger_Coburn1987} states and proves the following estimate, which holds with the same proof for general $t > 0$:
\[ \big \| \widetilde{|g|^2}^{(t)} \big \|_\infty \leq \big \| T_{|g|^2}^{t} \big \| \leq 4^n \big \| \widetilde{|g|^2}^{(t)} \big \|_\infty. \]
If $f \geq 0$, letting $g = \sqrt{f}$ and applying this inequality proves that $T_f^t$ is bounded if and only if $\widetilde{f}^{(t)}$ is bounded. A simple estimate yields $\widetilde{f}^{(t)}(z) \geq \frac{1}{2^n} \widetilde{f}^{(\frac{t}{2})}(z)$. Further, $\widetilde{f}^{(t)}(z) \leq \| \widetilde{f}^{(\frac{t}{2})}\|_\infty$ can be seen to hold by the semigroup property of the heat transform. Therefore, $\widetilde{f}^{(t)}$ is bounded if and only if $\widetilde{f}^{(\frac{t}{2})}$ is bounded.
\end{proof}
The second result is about symbols with certain oscillatory behaviour at infinity. For $z \in \mathbb C^n$ let $\tau_z(w) = z-w, ~ w \in \mathbb C^n$. 
For $f \in L_{\textup{loc}}^1(\mathbb C^n)$, we say that $f$ is of \textit{bounded mean oscillation} and write $f \in \operatorname{BMO}$ if
\[ \sup_{z \in \mathbb C^n} \int_{\mathbb C^n} | f \circ \tau_z - \widetilde{f}^{(t)}(z)| d\mu_t < \infty. \]
This notion can be seen to be independent of $t > 0$, cf. \cite{Bauer_Coburn_Isralowitz2010} for details. As is known 
$\operatorname{BMO}$ contains unbounded function. One has: 
\begin{thm}[{Bauer-Coburn-Isralowitz \cite[Theorem 6]{Bauer_Coburn_Isralowitz2010}}]\label{BMOthm}
Let $f \in \operatorname{BMO}$. Then, $\widetilde{f}^{(t)}$ is bounded for one $t > 0$ if and only if it is bounded for all $t > 0$. 
In particular, $T_f^t: F_t^2 \to F_t^2$ is bounded if and only if $\widetilde{f}^{(\frac{t}{2})}$ is bounded.
\end{thm}
We will come back to the last result when discussing different characterizations of the Toeplitz algebra. 
%%%%%%%%%%%%%%%%%%%%%%%%%%%%%%%%%%%%%%%%%%%%%%%%%%%%%%%%%%%%%%%%%%%%%%%%%%%%%%%%%%%%%%%%%%%%%%%
\section{Characterizations of the Toeplitz algebra}
\label{Section_4}
%%%%%%%%%%%%%%%%%%%%%%%%%%%%%%%%%%%%%%%%%%%%%%%%%%%%%%%%%%%%%%%%%%%%%%%%%%%%%%%%%%%%%%%%%%%%%%
From now on we only consider the case $p = 2$ and $t=1$. It is tautological to say that operator-theoretic properties of Toeplitz operators are tightly related to properties of their symbols. 
One of the most basic results of this kind is the following: Let $f \in C(\mathbb C^n)$ be such that
\[ f(z) \to 0 \quad \text{\it as}  \quad |z| \to \infty. \]
Then, $T_f$ is compact. As is well-known the Berezin transform is one-to-one on the algebra  $\mathcal{L}(F_1^2)$ of all bounded operators. This indicates that operator theoretic properties are 
also tightly connected to properties of the Berezin transform, e.g. if $T_f$ is compact, then it holds
\[ \widetilde{T_f}(z) = \widetilde{f}(z) \to 0 \quad \mbox{\it as} \quad |z| \to \infty. \]
Conversely, it is an obvious question how compactness of an operators can be characterized in terms of the symbol (in case of a Toeplitz operator) or its Berezin transform. 
The first significant progress in this context was made for certain operators acting on the Bergman space over the unit disc $A^2(\mathbb D)$. We will not introduce all objects involved 
and refer to \cite{Zhu2007} for an introduction to Toeplitz operators on $A^2(\mathbb D)$.
\begin{thm}[{Axler-Zheng \cite[Theorem 2.2]{Axler_Zheng1998}}]
Let $A \in \mathcal L(A^2(\mathbb D))$ be a finite sum of finite products of Toeplitz operators with essentially bounded symbols. Then
\[ A \text{ is compact } \Longleftrightarrow \mathcal B(A)(z) \to 0 \; \;  \mbox{\it  as } z \to \partial \mathbb D. \]
Here, $\mathcal B(A)$ denotes the Berezin transform of $A$.
\end{thm}
The corresponding statement for the weighted Bergman spaces over the unit ball in $\mathbb{C}^n$ or even general bounded symmetric domains (with some extra conditions on the weights) 
was derived shortly afterwards \cite{Englis1999,Raimondo2000}. 
% and independently for the Bergman spaces $A^2(\mathbb B_n)$, where $\mathbb B_n$ denotes the Euclidean unit ball in $\mathbb C^n$, in \cite{Raimondo2000}. 
In \cite{Englis1999} also the Fock space is treated: 
\begin{thm}[{Engli\v{s} \cite[Theorem B]{Englis1999}}] \label{theorem_Englis_compactness_characterization}
Let $A \in \mathcal L(F_t^2)$, $t>0$ be a finite sum of finite products of Toeplitz operators with essentially bounded symbols. Then, the following are equivalent: 
\[ A \text{ is compact } \Longleftrightarrow \widetilde{A}(z) \to 0 \; \; \text{\it  as } |z| \to \infty. \]
\end{thm}
The next progress made towards a characterization of compactness for an even larger set of operators acting on $F_t^2$ followed several years later. 
Although we put $t=1$ in our discussion most of the remaining statements hold true in the more general case $t>0$ with almost identical proofs.

Set
\[ c_0 := 0, \quad c_{n+1} := \frac{1}{4(1-c_n)} \hspace{4ex} \mbox{\it where} \hspace{4ex} n \in \mathbb{N}_0\]
and note that the sequence $(c_n)_n$ is monotone increasing with $c_n< \frac{1}{2}$. Let $f: \mathbb C^n \to \mathbb C$ be measurable and define:
\[ \| f\|_{\mathcal D_{c_n}} := \underset{z \in \mathbb C^n}{\operatorname{ess ~sup}} |f(z) e^{-c_n|z|^2}|. \]
Then, letting
\[ \mathcal D_{c_n} =\big{\{} f: \mathbb C^n \to \mathbb C \:  \text{\it  measurable} \: : \:   \| f\|_{\mathcal D_{c_n}} < \infty\big{\}}, \]
the norm  $\| \cdot \|_{\mathcal D_{c_n}}$ induces the structure of a Banach space on $\mathcal D_{c_n}$. By varying $n \in \mathbb N_0$ we obtain an increasing scale of Banach spaces:
\[ L^\infty (\mathbb C^n) = \mathcal D_{c_0} \subset \mathcal D_{c_1} \dots \subset \mathcal D_{c_n} \subset \dots \subset \mathcal D := \bigcup_{n \in \mathbb N_0} \mathcal D_{c_n} \subset L_1^2. \]
Setting now
\[ \mathcal H_{c_n} = \mathcal D_{c_n} \cap \textup{Hol}(\mathbb C^n), \]
equipped with the norm of $\mathcal D_{c_n}$ we obtain a second scale of Banach spaces in the Fock space $F_1^2$:
\begin{equation}\label{Scale_of_Banach_spaces_in_the_Fock_space}
\mathbb C \cong \mathcal H_0 \subset \mathcal H_{c_1} \subset \dots \subset \mathcal H_{c_n} \subset \dots \subset \mathcal H := \bigcup_{n \in \mathbb N_0} \mathcal H_{c_n} \subset F_1^2.
\end{equation}
For $z \in \mathbb C^n$ we define the {\it Weyl operators} $W_z: F_1^2 \to F_1^2$ through
\[ W_z(f)(w) = k_z(w) f(w-z). \]
The operators $W_z$ are well known to be unitary. Further, it holds for all $z, w \in \mathbb C^n$:
\begin{enumerate}
\item $W_z^\ast = W_z^{-1} = W_{-z}$
\item $W_z W_w = e^{-i\im(z\cdot \overline{w})}W_{z+w}$. 
\end{enumerate}
For a linear operator $A$ on $F_1^2$ and $z \in \mathbb{C}$ we set
\[ A_z := W_z^\ast AW_z. \]
\begin{defn}\label{Definition_F_l_F_al}
A bounded $\mathbb R$-linear operator $A: F_1^2 \to F_1^2$ with $A(\mathcal H) \subset \mathcal H$ is said to \textit{act uniformly continuously} on the scale (\ref{Scale_of_Banach_spaces_in_the_Fock_space}) if 
for each $n_1 \in \mathbb N_0$ there exists $n_2 \geq n_1$ and $d>0$ independent of $z \in \mathbb{C}^n$ such that for all $ f \in \mathcal{H}_{n_1}$: 
\[ \| A_z f\|_{\mathcal D_{c_{n_2}}} \leq d \| f\|_{\mathcal D_{c_{n_1}}}.
\]% \hspace{6ex} f \in \mathcal{H}_{n_1}, \; z \in \mathbb{C}^n. \]
We denote the spaces of $\mathbb C$-linear operators and the $\mathbb C$-antilinear operators acting uniformly continuously on the scale (\ref{Scale_of_Banach_spaces_in_the_Fock_space}) 
by $\mathcal F^l$ and $\mathcal F^{al}$, respectively.
\end{defn}
Several properties of $\mathcal F^l$ and $\mathcal F^{al}$ are known:
\begin{prop}[Bauer-Furutani \cite{Bauer_Furutani}]
\begin{enumerate}
\item $\mathcal F^l$ is a $\ast$-algebra,
\item $\mathcal F^l$ contains all Toeplitz operators with essentially bounded symbol,
\item $\mathcal F^{al}$ contains the operators $w_f$.
\end{enumerate}
\end{prop}
Here, for $f \in L^\infty(\mathbb C^n)$, we define the antilinear operator $w_f$ on $F_1^2$ by
\[ w_f(g)  := P^1(f Cg), \]
where $Cg(z) = \overline{g(z)}$ denotes complex conjugation. Those operators are closely related to the little Hankel operators. In \cite{Bauer_Furutani} they played a role 
in compactness characterizations of Toeplitz operators on the pluriharmonic Fock space. Since $w_f$ is not a finite sum of finite products of Toeplitz operators, the following 
result was an improvement (in the case $t=1$) of Theorem \ref{theorem_Englis_compactness_characterization}: 
\begin{thm}[{Bauer-Furutani \cite[Theorem 3.11]{Bauer_Furutani}}]\label{thm_Bauer_furutani_compactness_characterization}
Let $A \in \mathcal F^l \cup \mathcal F^{al}$. Then, $A$ is compact if and only if $\widetilde{A}(z) \to 0$ as $|z| \to \infty$.
\end{thm}

Recall that the (full) Toeplitz algebra is the $C^\ast$ algebra generated by Toeplitz operators with bounded symbols, i.e. 
\[ \mathcal T = C^\ast\big{(}\{ T_f \: :\:  ~ f \in L^\infty(\mathbb C^n)\}\big{)}, \]
where $C^\ast(M)$ denotes the $C^\ast$ algebra generated by a given set $M \subset \mathcal L(F_1^2)$.
\vspace{1mm}\par 
The next step forward was a complete characterization of compact operators in terms of the Toeplitz algebra and the Berezin transform. This result was first obtained in 
\cite[Theorem 9.5]{Suarez2007} and \cite[Theorem 1.1]{Mitkovski_Suarez_Wick2013} in the setting of the Bergman space and standard weighted Bergman space 
over the unit ball $\mathbb{B}_n$ in $\mathbb{C}^n$, respectively. There also is a version for the $p$-Fock space which we state next in the 
special case $p=2$. 
\begin{thm}[{Bauer-Isralowitz \cite[Theorem 1.1]{Bauer_Isralowitz2012}}]\label{thmcpt}
Let $A \in \mathcal L(F_1^2)$. Then, it holds
\[ A ~ \text{is compact} ~ \Longleftrightarrow A \in \mathcal T \text{ and } \widetilde{A}(z) \to 0 ~ \text{as} ~ |z| \to \infty. \]
\end{thm}
\begin{rem}
Theorem \ref{thmcpt} and its versions on Bergman spaces over $\mathbb B_n$ have been proven in the general setting of standard weighted $p$-Bergman spaces $A_\alpha^p(\mathbb B_n)$ 
over the unit ball $\mathbb{B}_n \subset \mathbb C^n$ (here the parameter $\alpha$ refers to the weight) and $p$-Fock spaces $F_t^p$ for $1 < p < \infty$ in the original papers. 
For simplicity, we have only quoted the Hilbert space result. Moreover, a version of this theorem on $p$-Bergman spaces over bounded symmetric domains has recently been studied in \cite{Hagger2019}.
\end{rem}
One could think that this is the end of the story for characterizing compact operators on $F_1^2$. However, there is a link between Theorem \ref{thm_Bauer_furutani_compactness_characterization}
 and Theorem \ref{thmcpt} which we will discuss next. If a priori an operator $A$ is not given as a finite sum of finite products of Toeplitz operators it may be difficult to either check whether 
 $A \in  \mathcal F^l \cup \mathcal F^{al}$ or $A \in \mathcal{T}$ as is assumed in the above theorems. On the other hand, there are well-known examples in the literature of operators 
 $A \in \mathcal L(F_1^2)$ which are not compact but with $\widetilde{A}(z) \to 0$ as $|z| \to \infty$. For completeness we will present such an example which even is a bounded 
 Toeplitz operator (with unbounded symbol) below. Hence some additional assumptions  
 are required to ensure that vanishing of the Berezin transform at infinity implies compactness of the operator. It is therefore desirable to extend Theorem \ref{thmcpt} to a class 
 $\mathcal A \subset \mathcal L(F_1^2)$ containing the Toeplitz algebra and for which membership $A \in \mathcal{A}$ is easier to check. On the Fock space a new approach 
 appeared in \cite{Isralowitz_Mitkovski_Wick,Xia_Zheng}:
\begin{defn}
An operator $A \in \mathcal L(F_1^2)$ is said to be \textit{sufficiently localized} if there exist constants $\beta, C$ with $2n < \beta < \infty$ and $0 < C$ such that
\begin{equation}\label{sufficiently_localized}
\big{|}\big{\langle} Ak_z, k_w\big{\rangle}_1\big{|}\leq \frac{C}{(1+|z-w|)^\beta}.
\end{equation}
We denote the set of all sufficiently localized operators by $\mathcal A_{sl}$.
\end{defn}
We note at this point that boundedness of an operator $A$ with $k_z \in D(A)$ for all $z \in \mathbb{C}^n$ and with (\ref{sufficiently_localized}) already follows under certain natural 
assumptions on the domain of $A^\ast$:
\begin{lem}\label{sufficiently_localized_is_bounded}
Let $A$ be a densely defined operator on $F_1^2$ such that 
\[\Span \{ K_z \: :\:  ~ z \in \mathbb C^n\} \subset D(A)\cap D(A^*).\]
If there is a positive function $H \in L^1(\mathbb C^n,V)$ with 
\[ |\langle Ak_z, k_w\rangle_1| \leq H(z-w) \]
for all $z, w \in \mathbb C^n$, then $A$ is bounded. In particular, (\ref{sufficiently_localized}) implies boundedness of $A$ since 
$$\frac{1}{(1+|z|)^\beta} \in L^1(\mathbb C^n,V) \hspace{3ex}\mbox{\it if} \hspace{3ex}  \beta > 2n.$$
\end{lem}
\begin{proof}
Similarly to the computation in (\ref{GL_representation_Toeplitz_operator_as_integral}) one obtains:
\begin{align*}
(Af)(w) 
&= \int_{\mathbb C^n} f(z) \overline{(A^\ast K_w)(z)} d\mu_1(z)\\
&= \int_{\mathbb C^n} \| K_z\|_1 \|K_w\|_1 \big{\langle} Ak_z, k_w\big{\rangle}_1 f(z) d\mu_1(z).
\end{align*}
%We can further compute the kernel as
%\begin{align*}
%\overline{(A^\ast K_w)(z)} = \overline{ \langle A^\ast K_w, K_z\rangle} = \langle A K_z, K_w\rangle = \| K_z\| \|K_w\| \langle Ak_z, k_w\rangle.
%\end{align*}
Letting
\[ \widetilde{A}(z,w) := \big{\langle} Ak_z, k_w \big{\rangle}_1, \]
this implies
\[ (Af)(w) = \int_{\mathbb C^n} e^{\frac{|z|^2+|w|^2}{2}} \widetilde{A}(z,w) f(z) d\mu_1(z). \]
Observe that it suffices to prove boundedness of the integral operator with kernel $|\widetilde{A}(z,w)|e^{\frac{|z|^2 + |w|^2}{2}}$. It holds
\begin{align*}
\frac{1}{\pi^n} \int_{\mathbb C^n} |\widetilde{A}(z,w)|&e^{\frac{|z|^2 + |w|^2}{2}} e^{\frac{|w|^2}{2}} e^{-|w|^2} dV(w)= \\
&= e^{\frac{|z|^2}{2}}\frac{1}{\pi^n} \int_{\mathbb C^n} |\widetilde{A}(z,w)| dV(w)\\
&\leq e^{\frac{|z|^2}{2}} \frac{1}{\pi^n} \int_{\mathbb C^n} H(z-w) dV(w)\\
&= \frac{1}{\pi^n}\| H\|_{L^1(\mathbb C^n,V)} e^{\frac{|z|^2}{2}}. 
\end{align*}
Analogously: 
\begin{align*}
\frac{1}{\pi^n} \int_{\mathbb C^n} |\widetilde{A}(z,w)| e^{\frac{|z|^2 + |w|^2}{2}} e^{\frac{|z|^2}{2}} e^{-|z|^2} dV(z) \leq \frac{1}{\pi^n}\| H\|_{L^1(\mathbb C^n,V)} e^{\frac{|w|^2}{2}}.
\end{align*}
Using Schur's test \cite[Theorem 5.2]{Halmos_Sunder1978} with the function $h(w) = e^{\frac{|w|^2}{2}}$, one therefore obtains
\[ \| A\| \leq \frac{1}{\pi^n} \| H\|_{L^1(\mathbb C^n,V)}. \]
For the particular case
\[ H(w) := \frac{1}{(1+|w|)^\beta} \]
with $\beta > 2n$, one easily sees that $H \in L^1(\mathbb C^n,V)$ using polar coordinates.
\end{proof}
Checking localizedness of an operator in the above sense is indeed simpler than checking membership in the Toeplitz algebra. 
In fact, (\ref{sufficiently_localized}) reduces (in principle) to some integral estimate. Further, any Toeplitz operator with bounded symbol is indeed sufficiently localized 
according to the simple estimate in Lemma \ref{lmm2}. Hence we obtain:
\[ \mathcal T \subseteq C^\ast(\mathcal A_{sl}). \]
We now relate the notion of localized operators to the result in Theorem \ref{thm_Bauer_furutani_compactness_characterization}. 
%With the notation in Definition \ref{Definition_F_l_F_al} one has:
\begin{lem}\label{lmmactinguniformly}
Let $A \in \mathcal F^l$. Then, $A$ is sufficiently localized.
\end{lem}
\begin{proof}
With $z, w \in \mathbb C^n$ we obtain:
\begin{align*}
\langle A k_z, k_w\rangle_1 &= \langle A W_z 1, W_z W_z^\ast W_w 1\rangle_1\\
&= \langle A_z 1, W_{-z} W_w 1\rangle_1 \\
&= \langle A_z 1, W_{w-z}1\rangle_1 e^{-i \im(z \cdot \overline{w})}\\
&= \langle A_z 1, k_{w-z}\rangle_1 e^{-i\im(z \cdot \overline{w})}.
\end{align*}
Using this, we obtain for a suitable index $k \in \mathbb N_0$ and from $\| 1\|_{\mathcal D_{c_k}} = 1$:
\begin{align*}
|\langle A k_z, k_w\rangle_1| 
&\leq \frac{1}{\pi^n}\int_{\mathbb C^n} \big{|}[A_z1](u) k_{w-z}(u)\big{|} e^{-|u|^2} dV(u)\\
&\leq \frac{1}{\pi^n}\int_{\mathbb C^n} \| A_z1 \|_{\mathcal D_{c_k}} e^{\re((w-z)\cdot \overline{u}) - \frac{|w-z|^2}{2} - (1-c_k)|u|^2} dV(u)\\
&\leq \frac{d}{\pi^n}\int_{\mathbb C^n} e^{\re((w-z)\cdot \overline u) - (1-c_k)|u|^2} dV(u)~ e^{-\frac{|w-z|^2}{2}},
\end{align*}
where $d>0$ is the constant from the uniform continuous action of $A$ on the scale (\ref{Scale_of_Banach_spaces_in_the_Fock_space}).
Recall that for $a \in \mathbb C^n$, $\gamma > 0$ and by applying the properties of the reproducing kernel:
\begin{align*}
e^{\gamma|a|^2} &= \int_{\mathbb C^n} |e^{\gamma a \cdot \overline{u}}|^2 d\mu_{1/\gamma}(u)\\
&= \int_{\mathbb C^n} e^{2\gamma \re( a\cdot \overline{u})} d\mu_{1/\gamma}(u).
\end{align*}
Letting $\gamma = 1-c_k$ and $a = \frac{w-z}{2\gamma}$ gives
\begin{align*}
\int_{\mathbb C^n} e^{\re((w-z)\cdot \overline{u}) - (1-c_k)|u|^2} dV(u) &= \frac{\pi^n}{(1-c_k)^n} e^{\frac{|z-w|^2}{4(1-c_k)}}.
\end{align*}
We hence obtain
\begin{align*}
\big{|}\big{\langle} A k_z, k_w\big{\rangle}_1\big{|} &\leq \frac{d}{(1-c_k)^n}e^{\frac{|z-w|^2}{4(1-c_k)} - \frac{|z-w|^2}{2}}\\
&= \frac{d}{(1-c_k)^n}e^{-(\frac{1}{2} - \frac{1}{4(1-c_k)}){|z-w|^2}} \\
&= \frac{d}{(1-c_k)^n}e^{-(\frac{1}{2} - c_{k+1})|z-w|^2}.
\end{align*}
Since $c_{k+1} < \frac{1}{2}$ the statement follows.
\end{proof}
%\begin{cor} \label{Corollary_boundedness_of_operators_in_F_l}
%Let $A\in \mathcal{F}^l$, then $A$ is densely defined on $D(A):=\mathcal{H}$ in (\ref{Scale_of_Banach_spaces_in_the_Fock_space}) and it extends to a bounded operator on $F_1^2$. 
%\end{cor}
%\begin{proof}%
%According to Lemma \ref{lmmactinguniformly} the operator $A$ is sufficiently local
%\end{proof}
We can slightly relax the boundedness assumption of the operator symbol and still obtain sufficiently localized Toeplitz operators. More precisely: 
\begin{lem}\label{lmmbmosymbol}
Let $f \in \operatorname{BMO}$ be such that the Toeplitz operator $T_f$ is bounded. Then, $T_f$ is sufficiently localized.
\end{lem}
\begin{proof}
According to Lemma \ref{lmm1} and for $0 < s < 1$ we have the identity
\[ \big{\langle} f k_z, k_w \big{\rangle}_1 = e^{\frac{s}{2(1-s)}|z-w|^2 - \frac{is}{(1-s)}\im(w \cdot \overline z)} \big{\langle} \widetilde{f}^{(s)} k_z^{1-s}, k_w^{1-s}\big{\rangle}_{1-s}. \]
Applying Lemma \ref{lmm2} with $g \in L^\infty(\mathbb C^n)$ shows
\[ \big{|}\big{\langle} gk_z^{1-s}, k_w^{1-s}\big{\rangle}_{1-s}\big{|} \leq \| g\|_\infty e^{-\frac{1}{4(1-s)}|z-w|^2}. \]
By Theorem \ref{BMOthm} the heat transform $\widetilde{f}^{(s)}$ is bounded for each $0 < s < 1$. Choosing now $s<1/2$, we obtain from the above estimates
\begin{align*}
\big{|}\big{\langle} f k_z, k_w\big{\rangle}_1\big{|} \leq \| \widetilde{f}^{(s)}\|_\infty e^{\frac{2s-1}{4(1-s)}|z-w|^2},
\end{align*}
where $\frac{2s-1}{4(1-s)}<0$.
\end{proof}
It was proven by J. Xia and D. Zheng that $C^\ast(\mathcal A_{sl})$ is not too large in the sense that it still allows the desired compactness characterization:
\begin{thm}[{Xia-Zheng \cite[Theorem 1.2]{Xia_Zheng}}]\label{thm_xia_zheng}
Let $A \in C^\ast(\mathcal A_{sl})$. Then, $A$ is compact if and only if $\widetilde{A}(z) \to 0$ as $z \to \infty$.
\end{thm}
At this point we want to mention that the compactness characterization was already known to hold for bounded single Toeplitz operators with (possibly unbounded) 
symbol in $\operatorname{BMO}$. This has been achieved by N. Zorboska in \cite{Zorboska2003} for the Bergman space $A^2(\mathbb D)$ and later by L. Coburn, J. Isralowitz 
and B. Li in the setting of the Fock space  \cite{Coburn_Isralowitz_Li2011}. By combining Lemma \ref{lmmbmosymbol} and Theorem \ref{thm_xia_zheng}, we obtain a 
generalization of the latter result: 
\begin{cor}\label{Corollary_Coburn_Isralowitz_Li}
Let $A \in \mathcal L(F_1^2)$ be a finite sum of finite products of bounded Toeplitz operators with (possibly unbounded) symbols in $\operatorname{BMO}$. Then, $A$ is compact 
if and only if $\widetilde{A}(z) \to 0$ as $|z| \to \infty$.
\end{cor}
A next step was taken in \cite{Isralowitz_Mitkovski_Wick}:
\begin{defn}
An operator $A \in \mathcal L (F_1^2)$ is called {\it weakly localized} if:% the following hold:
\begin{align*}
\sup_{z \in \mathbb C^n} &\int_{\mathbb C^n} |\langle Ak_z, k_w\rangle_1| dV(w) < \infty, \; \\
\sup_{z \in \mathbb C^n} &\int_{\mathbb C^n} |\langle A^\ast k_z, k_w\rangle_1| dV(w) < \infty, \\
 \lim_{r \to \infty} \sup_{z \in \mathbb C^n} &\int_{|z-w|\geq r} |\langle A k_z, k_w\rangle_1|dV(w) = 0,\\
%\sup_{z \in \mathbb C^n} \int_{\mathbb C^n} |\langle A^\ast k_z, k_w\rangle_1| dV(w) < \infty, \; 
 \lim_{r \to \infty} \sup_{z \in \mathbb C^n} &\int_{|z-w|\geq r} |\langle A^\ast k_z, k_w\rangle_1|dV(w) = 0. 
\end{align*}
The set of all weakly localized operators is denoted by $\mathcal A_{wl}$.
\end{defn}
It is easy to see that $\mathcal A_{sl} \subseteq \mathcal A_{wl}$. The Fock space analogue of Theorem \ref{thmcpt} for the setting of $C^\ast(\mathcal A_{wl})$ was not completely 
proven in \cite{Isralowitz_Mitkovski_Wick}.  The authors gave a proof of the corresponding theorem for the Bergman spaces over the unit ball.%, but there was some work into that direction in that paper.

Finally, the following surprising result was obtained by J. Xia \cite{Xia} in the setting of the Bergman space $A^2(\mathbb B_n)$ as well as for the Fock space $F_1^2$:
\begin{thm}[{Xia \cite[Section 4]{Xia}}]\label{thmxia}
It holds
\[ \mathcal T^{(1)} = \mathcal T = C^\ast(\mathcal A_{sl}) = C^\ast(\mathcal A_{wl}). \]
\end{thm}
Here, $\mathcal T^{(1)}$ denotes the norm closure of $\{ T_f\: :\: ~ f \in L^\infty\}$. Theorem \ref{thmxia} indicates that indeed all the generalizations of Theorem \ref{thmcpt} 
were just new characterizations of the Toeplitz algebra. Based on the observations in Lemma \ref{lmmactinguniformly} and Lemma \ref{lmmbmosymbol} we can add two other characterizations:
\begin{thm}\label{theorem_new_characterization}
Denote by $\overline{\mathcal F^l}$ the operator norm closure of $\mathcal F^l$, which is a $C^\ast$ algebra. It then holds:
\begin{align*}
%\mathcal T^{(1)} &= \mathcal T = C^\ast(\mathcal A_{sl}) = C^\ast(\mathcal A_{wl})\\
\mathcal T= \overline{\mathcal F^l}= C^\ast\big{(}\{ T_f; ~ f \in \operatorname{BMO}\: \text{ and } \:  T_f \text{ bounded}\}\big{)}. 
\end{align*}
\end{thm}
\begin{proof}
Obviously, $\overline{\mathcal F^l}$ defines a $C^*$ algebra in $\mathcal{L}(F_1^2)$ and since all Toeplitz operators with essentially bounded symbols are uniformly 
continuously acting on the scale  (\ref{Scale_of_Banach_spaces_in_the_Fock_space}) (cf. \cite{Bauer_Furutani} for a calculation) it follows that $\mathcal{T} \subset \overline{\mathcal F^l}$. 
Since elements in $\mathcal{F}^l$ are sufficiently localized according to Lemma \ref{lmmactinguniformly} we have: 
\begin{equation*}
\overline{\mathcal{F}^l} \subset C^*(\mathcal{A}_{sl})= \mathcal{T}. 
\end{equation*}
The last identity follows from J. Xia's result in Theorem \ref{thmxia}. The assertion $\mathcal T = \overline{\mathcal F^l}$ is obtained by combining both inclusions. 

As for the second characterization of the Toeplitz algebra recall that $L^\infty \subset \operatorname{BMO}$. It follows, using Lemma \ref{lmmbmosymbol}
\[ \mathcal T \subset C^\ast\big{(}\{ T_f; ~ f \in \operatorname{BMO} \: \text{\it  and } \: T_f\:  \text{\it  bounded} \}\big{)} \subset C^\ast(\mathcal A_{sl}). \]
Equality again follows from Theorem \ref{thmxia}.
\end{proof}
It is worth mentioning, that there are bounded Toeplitz operators (necessarily with symbols not in $\textup{BMO}$) which do not satisfy the desired compactness characterization.
 In particular, they cannot define elements in the Toeplitz algebra. We give one such example (which has been previously known in the literature): 
\begin{ex}
Consider again the function $g_\lambda(z)=e^{\lambda |z|^2}$ from Example \ref{ex1}. Let $\lambda \in \mathbb C$ be such that
\begin{align}
\re(\lambda) &< \frac{1}{2} \label{lambda1}\\
|1-\lambda| &= 1 \label{lambda2}\\
|1-2\lambda| &> 1. \label{lambda3}
\end{align}
It is easy to see that such $\lambda$ exist. Assumption (\ref{lambda1}) guarantees that $\widetilde{T_{g_\lambda}}$ is well-defined. Recall that it is given by
\[ \widetilde{T_{g_\lambda}}(z) = \frac{1}{(1-\lambda)^n} e^{\frac{\lambda}{1-\lambda}|z|^2}. \]
Assumption (\ref{lambda3}) is equivalent to $\re(\lambda/(1-\lambda)) < 0$. Hence, it implies that 
\[ \widetilde{T_{g_\lambda}}(z) \to 0 \hspace{4ex} \mbox{\it  as } \hspace{4ex} |z| \to \infty. \] 
Further, recall that $T_{g_\lambda}$ acts on the standard orthonormal basis in (\ref{Standard_ONB_Fock_space}) as:
\[ T_{g_\lambda} e_m^1 = \frac{1}{(1-\lambda)^{|m|+n}} e_m^1,  \hspace{4ex}m \in \mathbb{N}_0^n. \]
Put $\nu = \frac{1}{1-\lambda}$. Since $e_m^1$ is a homogeneous polynomial of degree $|m|$ we have: 
\[ T_{g_\lambda} e_m^1(z) = \nu^n \cdot e_m^1(\nu z). \]
As assumption (\ref{lambda2}) implies that $T_{g_\lambda}$ is bounded, this equation extends to all of $F_1^2$:
\[ T_{g_\lambda}f(z) =\nu^n \cdot f(\nu z), \hspace{4ex}  f \in F_1^2. \]
Because of $|\nu| = 1$ the Toeplitz operator $T_{g_\lambda}$ is unitary, hence cannot be compact.
\end{ex}
Let us return to the initial proof of Theorem \ref{thmcpt} in \cite{Bauer_Isralowitz2012} (which is similar to the proof in the case of the Bergman space 
$A^2(\mathbb B_n)$ \cite{Suarez2007, Mitkovski_Suarez_Wick2013}). Among other ideas the arguments used results on limit operators. Based on such limit operator techniques, a characterization 
of the Fredholm property of operators in $\mathcal T$ has been proven in \cite{Fulsche_Hagger} (subsequently to \cite{Hagger2017} which contains the analysis in the case of the 
standard weighted Bergman spaces $A_\alpha^p(\mathbb B_n)$). There, the so-called band-dominated operators were introduced:
\begin{defn}
\begin{enumerate}
\item An operator $A \in \mathcal L(L_1^2)$ is called a {\it band operator} if there is a number $\omega > 0$ such that $M_f AM_g = 0$ for all $f, g \in L^\infty(\mathbb C^n)$ with 
$$d\big{(}\operatorname{supp} f, \operatorname{supp} g\big{)} > \omega.$$ 
Here $d$ denotes the Euclidean distance and we write $M_f$ for the multiplication operator by $f$. The infimum of all such $\omega$ will be called the \emph{band-width} of $A$.
\item  The set $\operatorname{BDO}$ of {\it band-dominated operators} is defined as the norm-closure of all band operators in $\mathcal L(L_1^2)$.
\end{enumerate}
\end{defn}
The following properties of $\operatorname{BDO}$ have been derived, which relate compressions to $F_1^2$ of band-dominated operators to the Toeplitz algebra.
\begin{prop}[{Fulsche-Hagger \cite{Fulsche_Hagger}}] \begin{enumerate}
\item $\operatorname{BDO}$ is a $C^\ast$ algebra of operators on $L_1^2$ containing $P$ and  multiplication operators $M_f$ for $f \in L^\infty$.
\item $P \operatorname{BDO} P \subset \mathcal L(F_1^2)$ contains $\mathcal T$. 
\end{enumerate}
\end{prop}
In principle, a similar limit operator approach as in \cite{Bauer_Isralowitz2012} can be used to derive a compactness characterization for operators from $P \operatorname{BDO}P$. 
This has not been worked out for the Fock space, but was done for the Bergman space over bounded symmetric domains in \cite[Theorem A]{Hagger2019}. After the preceding discussions, this naturally leads to the question whether $P \operatorname{BDO} P = \mathcal T$. That this is indeed true will be the content of the last part of this work. Lemma \ref{BO_lemma} and the main result in Theorem 
\ref{Theorem_ T_=_PBDOP} were communicated to us by Raffael Hagger. 
\begin{lem}\label{BO_lemma}
For all $z \in \mathbb C^n$ and $r>0$ it holds
\[ \| M_{1-\chi_{B(z,r)}} k_z\|_{L_1^2} \leq C_n e^{-\frac{r^2}{2n}}, \]
where $C_n> 0$ is some constant depending only on the dimension $n$. Here, $B(z,r) \subset \mathbb C^n$ is the Euclidean ball of radius $r$ around $z$ and $\chi_{B(z,r)}$ is the indicator function of that ball.
\end{lem}
\begin{proof}
For $z \in \mathbb C^n$ and $r > 0$ denote by $Q(z,r)$ the set
\begin{align*}
Q(z,r) :&= \{ w \in \mathbb C^n ~ : ~|w_j - z_j| < r \text{ \it for all } j = 1, \dots, n\}\\
&= \prod_{j=1}^n D(z_j, r),
\end{align*}
where $D(z_j, r) \subset \mathbb C$ is the disc around $z_j \in \mathbb C$ of radius $r$. It is immediate that $Q(0, \frac{r}{\sqrt{n}}) \subset B(0,r)$.
 
We estimate the norm for $z = 0$ first:
\begin{align*}
\| M_{1-\chi_{B(0,r)}}k_0\|_{L_1^2}^2 &= 1 - \frac{1}{\pi^n}\int_{B(0,r)} e^{-|z|^2} dV(z)\\
&\leq 1 - \frac{1}{\pi^n}\int_{Q(0, \frac{r}{\sqrt{n}})} e^{-|z|^2} dV(z)\\
&= 1 - \left( \frac{1}{\pi} \int_{D(0, \frac{r}{\sqrt{n}})} e^{-|z_1|^2} dV_1(z_1) \right)^n,
\end{align*}
where in the last integral $V_1$ denotes the Lebesgue measure on $\mathbb C$. Using polar coordinates, one easily sees that
\begin{align*}
\frac{1}{\pi} \int_{D(0, \frac{r}{\sqrt{n}})} e^{-|z_1|^2} dV_1(z_1) = 1 - e^{-\frac{r^2}{n}}.
\end{align*}
This gives
\begin{align*}
\| M_{1-\chi_{B(0,r)}}k_0\|_{L_1^2} &\leq 1 - \left( 1 - e^{-\frac{r^2}{n}}\right)^n = 1 - \sum_{k=0}^n \binom{n}{k} (-1)^k e^{-\frac{kr^2}{n}}\\
&=\sum_{k=1}^n \binom{n}{k} (-1)^{k+1} e^{-\frac{kr^2}{n}} \leq \sum_{k=1}^n \binom{n}{k} e^{-\frac{r^2}{n}}.
\end{align*}
For general $z \in \mathbb C^n$, we obtain
\begin{align*}
&\| M_{1-\chi_{B(z,r)}} k_z\|_{L_1^2} = \| W_{-z} M_{1-\chi_{B(z,r)}} W_z k_0\|_{L_1^2}\\
&\quad = \| M_{1-\chi_{B(0,r)}} k_0\|_{L_1^2} \leq C_n e^{-\frac{r^2}{2n}},
\end{align*}
where we used the facts that $W_{-z}$ is an isometry and $k_z = W_z k_0$ in the first equality and $W_{-z} M_f W_z = M_{f \circ \tau_{-z}}$ for arbitrary functions $f \in L^\infty(\mathbb C^n)$ in the second equality.
\end{proof}
\begin{thm}\label{Theorem_ T_=_PBDOP}
It holds $\mathcal T = P \operatorname{BDO} P$.
\end{thm}
\begin{proof}
It suffices to prove that $P \operatorname{BDO} P \subset \mathcal T$. We will do this by proving that $P \operatorname{BO} P \subset \mathcal A_{sl}$, the result then follows from Xia's Theorem \ref{thmxia}. 

Let $A \in \operatorname{BO}$ have band-width $\omega$. Let $z, w \in \mathbb C^n$ be such that $|z-w| \leq 3\omega$. Then,
\begin{align*}
|\langle PAPk_z, k_w\rangle| \leq \| A\| = \| A\| e^{\frac{\omega^2}{2n}} e^{-\frac{\omega^2}{2n}} \leq \| A\| e^{\frac{\omega^2}{2n}} e^{-\frac{|z-w|^2}{18n}}.
\end{align*}
For $|z-w| > 3\omega$ we put $r = \frac{|z-w|}{3}$. Observe that this implies 
\begin{equation}\label{distance_balls}
d(B(z,r), B(w,r)) > \omega.
\end{equation}
We obtain
\begin{align*}
|\langle PAPk_z, k_w\rangle | &= |\langle Ak_z, k_w\rangle|\\
&= \big | \langle A M_{\chi_{B(z,r)}}k_z, M_{\chi_{B(w,r)}} k_w\rangle + \langle AM_{1-\chi_{B(z,r)}} k_z, M_{\chi_{B(w,r)}} k_w\rangle \\
&\quad + \langle Ak_z, M_{1-\chi_{B(w,r)}} k_w\rangle \big |.
\end{align*}
The first term vanishes by Equation (\ref{distance_balls}). For the other two terms, we apply Lemma \ref{BO_lemma} and obtain the estimate 
\begin{align*}
|\langle PAPk_z, k_w\rangle| \leq 2C_n\| A\| e^{-\frac{|z-w|^2}{18n}}.
\end{align*}
Adjusting the constants, we obtain a uniform estimate for $|\langle PAPk_z, k_w\rangle|$ for all $z, w \in \mathbb C^n$ which proves that $A$ is sufficiently localized.
\end{proof}
Since many of the introduced objects also exist in the setting of $p$-Fock spaces and several of the mentioned results carry over to this setting one may also ask: 
\begin{question}
Is there an $F_t^p$-analogue of Theorem \ref{thmxia}?
\end{question}
\subsection*{Acknowledgment}
We wish to thank Raffael Hagger who has communicated to us Lemma \ref{BO_lemma} and Theorem \ref{Theorem_ T_=_PBDOP}. 
%%%%%%%%%%%%%%%%%%%%%%%%%%%%%%%%%%%%%%%%%%%%%%%%%%%%%%%%%%%%%%%%%%%%%%%%%%%%%%%%%%%%%%%%%
%\subsection*{Acknowledgment}
%Many thanks to our \TeX-pert for developing this class file.

% ------------------------------------------------------------------------
\end{document}